\documentclass[10pt]{article}
\usepackage{fancyhdr}
\usepackage{graphicx} 
\usepackage{amsfonts}
\usepackage{amsthm,amssymb,amsmath}

\newcommand{\Oplus}{\ensuremath{\vcenter{\hbox{\scalebox{1.4}{$\oplus$}}}}}
\newcommand{\Otimes}{\ensuremath{\vcenter{\hbox{\scalebox{1.4}{$\otimes$}}}}}

\usepackage[T1]{fontenc}
\usepackage{mathabx}
\usepackage[utf8]{inputenc}
\usepackage{lipsum}
\usepackage{mathrsfs}
\usepackage{MnSymbol}
\usepackage{mathtools}
\usepackage{float}
\usepackage[all,cmtip]{xy}
\usepackage{tikz-cd}
\tikzcdset{row sep/normal=50pt, column sep/normal=50pt}
\newtheorem{lemma}{Lemma}[section]
\newtheorem{remark}[lemma]{Remark}
\newtheorem{theorem}[lemma]{Theorem}
\newtheorem{example}[lemma]{Example}
\newtheorem{proposition}[lemma]{Proposition}
\newtheorem{corollary}[lemma]{Corollary}

\newtheorem{manualtheoreminner}{Theorem}
\newenvironment{manualtheorem}[1]{%
  \renewcommand\themanualtheoreminner{#1}%
  \manualtheoreminner
}{\endmanualtheoreminner}
\usepackage{hyperref}
\usepackage{graphicx} 
\usepackage{amsfonts}
\usepackage{amsthm,amssymb,amsmath}
\usepackage[T1]{fontenc}
\usepackage{mathabx}
\usepackage[utf8]{inputenc}
\usepackage{hyperref}
\usepackage{lipsum}
\usepackage{mathrsfs}
\usepackage{enumitem}
\usepackage{MnSymbol}
\usepackage{stmaryrd}
\usepackage{mathtools}
\usepackage{float}
\usepackage[all,cmtip]{xy}
\usepackage{tikz-cd}
\tikzcdset{row sep/normal=50pt, column sep/normal=50pt}

\usepackage[a4paper, total={6in, 8in}]{geometry}
\usepackage{blindtext}

\makeatletter
\newcommand*{\da@rightarrow}{\mathchar"0\hexnumber@\symAMSa 4B }
\newcommand*{\da@leftarrow}{\mathchar"0\hexnumber@\symAMSa 4C }
\newcommand*{\xdashrightarrow}[2][]{%
\mathrel{%
\mathpalette{\da@xarrow{#1}{#2}{}\da@rightarrow{\,}{}}{}%
}%
}
\newcommand{\xdashleftarrow}[2][]{%
\mathrel{%
\mathpalette{\da@xarrow{#1}{#2}\da@leftarrow{}{}{\,}}{}%
}%
}
\newcommand*{\da@xarrow}[7]{%
\sbox0{$\ifx#7\scriptstyle\scriptscriptstyle\else\scriptstyle\fi#5#1#6\m@th$}%
\sbox2{$\ifx#7\scriptstyle\scriptscriptstyle\else\scriptstyle\fi#5#2#6\m@th$}%
\sbox4{$#7\dabar@\m@th$}%
\dimen@=\wd0 %
\ifdim\wd2 >\dimen@
\dimen@=\wd2 %
\fi
\count@=2 %
\def\da@bars{\dabar@\dabar@}%
\@whiledim\count@\wd4<\dimen@\do{%
\advance\count@\@ne
\expandafter\def\expandafter\da@bars\expandafter{%
\da@bars
\dabar@ 
}%
}%
\mathrel{#3}%
\mathrel{%
\mathop{\da@bars}\limits
\ifx\\#1\\%
\else
_{\copy0}%
\fi
\ifx\\#2\\%
\else
^{\copy2}%
\fi
}%
\mathrel{#4}%
}
\makeatother
\title{Correspondence between Projective bundles over $ \mathbb{P}^2 $ and rational Hypersurfaces in $\mathbb{P}^4$} 
\author{ Shivam Vats}
\date{}

\begin{document}
\maketitle{}
\begin{large}
\begin{center}
\textbf{Abstract}
\end{center}
Let $E$ be the restriction of the null-correlation bundle on $ \mathbb{P}^3$  to a hyperplane. In this article, we show that the projective bundle $ \mathbb{P}(E)$ is isomorphic to a blow-up of a non-singular quadric in $ \mathbb{P}^4 $ along a line. We also prove that for each   $d\geq2$, there are hypersurfaces of degree $d$ containing a line in $ \mathbb{P}^4$   whose blow-up along the line is isomorphic to the projective bundle over $ \mathbb{P}^2$.   
\\

\begin{center}
\end{center}

\section{Introduction}
Null-correlation bundle on $\mathbb{P}^3 $ belonged to 
the first known examples of indecomposable 2-bundles on $\mathbb{P}^3$. It is a rank 2 vector bundle on $\mathbb{P}^3$ with Chern classes $ c_{1}=0$ and $ c_{2}= 1.$ In \ref{2}, Barth showed that the null-correlation bundle is the only rank two vector bundle up to tensoring by a line bundle on $\mathbb{P}^3$, which is stable and whose restriction to every hyperplane in $ \mathbb{P}^3 $ is semistable but not stable. Later, a geometric construction of the null-correlation bundle is given in \ref{4}. In \ref{4}, they constructed a rank two bundle $Q$ on $ \mathbb{P}^3 $ and
showed that it is a stable bundle, but its restriction to any hyperplane is semistable but not stable, and also $Q(-1)$ is isomorphic to a null-correlation bundle. \\
 Let $E=Q\mid_{H} $, where $H$ is a hyperplane in $ \mathbb{P}^3$ and $ \mathbb{P}(E)$ be  projectivization of $E$. The tautological line bundle $\mathcal{O}_{\mathbb{P}(E)}(1)$  defines a non-degenerate mapping to $ \mathbb{P}^4 $. In section 2 of this article, we show that the image of this map is a non-singular quadric hypersurface $S$ in $ \mathbb{P}^4$.\\
 We also prove the following theorem with the notations from the previous paragraph.
\begin{manualtheorem}{1}[]\label{foo}
 The blow-up of $S$ along a line $\ell$ contained in $S$ is isomorphic to  $ \mathbb{P}(E)$. 
\end{manualtheorem}
Moreover, in section 3, we prove the following theorem,
  
\begin{manualtheorem}{2}[] \label{baz}
There is an open set $U$ of $\mathbb{P}(H^{0}(\mathcal{O}_{\mathbb{P}^2}(2)\Oplus \mathcal{O}_{\mathbb{P}^2}^{\bigoplus 2}(1))$  such that for each point $s \in U$ we have a rank two vector bundle ${V}_{s}$ over $ \mathbb{P}^2$. Let $\mathbb{P}(V_{s})$ be the projectivization of ${V}_{s}$ and $\mathcal{O}_{\mathbb{P}({V}_{s})}(1)$ be the tautological bundle over $\mathbb{P}({V}_{s})$. Then the  linear system  $ |\mathcal{O}_{\mathbb{P}({V}_{s})}(1) |$  defines a non-degenerate mapping  $\mathbb{P}({V}_{s}) \longrightarrow \mathbb{P}^ 4$
 whose image is a smooth quadric $F_{s}$ and $\mathbb{P}({V}_{s}) $ is isomorphic to blow up of $F_{s}$ along a line $\ell$,  where $\ell$ is a fixed line contained in $F_{s}$ for each $s \in U$.
 \end{manualtheorem}
In section 3, we have studied the following problem. \\
 It is always interesting to know when the blow-up of a projective variety along a projective subvariety is isomorphic to the projective bundle over some projective variety. In general, this is not always true. There are examples where it happens. Let $Z$ be the projective variety obtained from the projective space $\mathbb{P}^n$  by blowing up along a linear subspace $\Gamma \cong  \mathbb{P}^{m-1}$ then from Section  9.3.2 of \ref{5} we have $ Z\cong \mathbb{P}(\mathcal{G})$ where $\mathcal{G} \cong \mathcal{O}_{\mathbb{P}^{n-m}}(1)\bigoplus \mathcal{O}^{\bigoplus m}_{\mathbb{P}^{n-m}} $ is a rank $m+1$ vector bundle over $\mathbb{P}^{n-m}$. In \ref{12}, there are examples of non-linear subvarieties in $\mathbb{P}^{i} \medspace \medspace \medspace (i = 3, 4, 5 )$, such that blow-ups of $ \mathbb{P}^{i}$  along these non-linear subvarieties are projective bundles over $ \mathbb{P}^2$ and for more general example see \ref{8}.
   Theorem \ref{foo}  provides an example of a blow-up of a projective variety, which is not projective space but whose blow-up is a projective bundle over projective space.  Theorem \ref{baz} also provides infinitely many examples of such type.  \\       
 
As a consequence of Theorem \ref{baz}, we observe that there is a family of vector bundles over $ \mathbb{P}^2$ whose projectivization corresponds to a family of smooth quadrics containing a fixed line in $ \mathbb{P}^4$. \\ 

In section 4, we generalize the ideas of section 3 and obtain a family of vector bundles over $ \mathbb{P}^2$ whose projectivization corresponds to a family of singular rational hypersurfaces containing a fixed line in $\mathbb{P}^4$. We prove the following.\\  
\begin{manualtheorem}{3}[] \label{raz}
  For $n\geq 3$ there is an open set $U_{n}$ of the $\mathbb{P}(H^{0}(\mathcal{O}_{\mathbb{P}^2}(n)\Oplus \mathcal{O}_{\mathbb{P}^2}^{\bigoplus 2}(n-1))$  such that for each point $s \in U_{n}$ we have a rank two vector bundle ${V}_{n,s}$ over $ \mathbb{P}^2$. Let $\mathbb{P}({V}_{n,s})$ be the projectivization of ${V}_{n,s}$ and $\mathcal{O}_{\mathbb{P}({V}_{n,s})}(1)$ be the tautological bundle over $\mathbb{P}({V}_{n,s})$. Then  linear system  $ |\mathcal{O}_{\mathbb{P}({V}_{n,s})}(1) |$  defines a non-degenerate mapping  $\mathbb{P}({V}_{n,s}) \longrightarrow \mathbb{P}^ 4$ whose image  is a hypersurface $F_{n,s}$ of degree $n$ which is  singular along a line $\ell$  and $\mathbb{P}({V}_{n,s}) $ is isomorphic to blow up of $F_{n,s}$ along a line $\ell$,  where $\ell$ is a fixed line contained in $F_{n,s}$ for each $s\in U_{n}$.
\end{manualtheorem}

\section{Projectivization of Q restricted to a hyperplane in
$\mathbb{P}^3$}
 We follow the notations and conventions from \ref{9} and \ref{7}. 
Let  $Q$  be the vector bundle constructed in   \ref{4}. It is a rank two vector bundle  over $ \mathbb{P}^3$ with Chern classes
\begin{equation*}
c_{1}(Q)= 2[h]\qquad and \qquad c_{2}(Q)= 2[h^{2}]
\end{equation*}
where $[h]$ denotes the hyperplane class in $ \mathbb{P}^3$.
\\

Let $E=Q\mid_{H}$ where $H$ is a hyperplane in $ \mathbb{P}^3$ i.e., $H\cong \mathbb{P}^2$. The Chern classes of  $E$  are 
\begin{equation*}
c_{1}(E)= 2[h]\qquad and  \qquad 
c_{2}(E)= 2[h^{2}]
\end{equation*}
where $[h]$ denotes the hyperplane class in $ \mathbb{P}^2$.
\begin{lemma}\label{lemma:lemma 2.1}
The twisted bundle E(-1) has a unique global section \lq $s$\rq \medspace such that the zero locus of \lq $s$\rq \medspace has a single point $P\in \mathbb{P}^2$.
\end{lemma}
\begin{proof}
By [\ref{4}, Section 4] , $Q$ is globally generated so is $E$.  For a general  element $t \in  H^{0}(\mathbb{P}^2,E )$  the zero set  $(t)_{0}$ is a codimension two reduced subscheme of degree two of $\mathbb{P}^2$  i.e.,  $(t)_{0} = \{P_{1},P_{2}\}$ with $P_{1} , P_{2} \in \mathbb{P}^2$. By [\ref{2}, Lemma 3], we have an exact sequence 
\begin{equation*}
 0\longrightarrow \mathcal{O}_{\mathbb{P}^2} \longrightarrow E \longrightarrow \mathcal{I}_{\{P_{1},P_{2}\}}(2) \longrightarrow 0 .
 \end{equation*}
Tensoring the above sequence by $\mathcal{O}_{\mathbb{P}^2}(-1)$ and taking the cohomology we have,
$$
 0\longrightarrow H^{0}(\mathbb{P}^2, \mathcal{O}_{\mathbb{P}^2}(-1)) \longrightarrow  H^{0}(\mathbb{P}^2, E(-1)) \longrightarrow H^{0}(\mathbb{P}^2, \mathcal{I}_{\{P_{1},P_{2}\}}(1)) $$ 
 $$\longrightarrow H^{1}(\mathbb{P}^2,\mathcal{O}_{\mathbb{P}^2}(-1))\longrightarrow H^{1}(\mathbb{P}^2, E(-1))\longrightarrow ....
 $$
hence we get  
\begin{equation*}
H^{0}(\mathbb{P}^2, E(-1)) \cong H^{0}(\mathbb{P}^2, \mathcal{I}_{\{P_{1},P_{2}\}}(1)). 
\end{equation*}
Since   dim $ H^{0}(\mathbb{P}^2, \mathcal{I}_{\{P_{1},P_{2}\}}(1)) =1$, we conclude that dim $H^{0}(\mathbb{P}^2,E(-1)) =1$. By [\ref{4}, Theorem 5.1 ]  $E$ is semistable so is $E(-1)$ and $$c_{1}(E(-1))=0\quad,\quad c_{2}(E(-1))= [h^2].$$ Choose a unique nonzero section $ s \in H^{0}(\mathbb{P}^2, E(-1))$, we have the exact sequence  $$ 0\longrightarrow \mathcal{O}_{\mathbb{P}^2} \longrightarrow E(-1) \longrightarrow \mathscr{F} \longrightarrow 0 $$
where $\mathscr{F} $ is the cokernel. We want to show zero locus $(s)_{0}$ is a point $P\in \mathbb{P}^2$. If \lq$s$\rq \medspace would vanish on a curve, say with equation $f=0$,  where $f\in \Gamma(\mathcal{O}_{\mathbb{P}^2}(d))$ with $d \geqslant 1$. Then $s/f$ would embed $ \mathcal{O}_{\mathbb{P}^2}(d)$ as a subsheaf in $E(-1)$ i.e., we have $$0 \longrightarrow \mathcal{O}_{\mathbb{P}^2}(d))\longrightarrow E(-1) .$$ But this is not possible as $E(-1)$ is semistable, so \lq$s$\rq \medspace vanishes in codimension two and we know $c_{2}(E(-1))= [h^{2}]$ this implies $(s)_{0}=\{P\}$.
\end{proof}
\begin{remark}\label{remark:remark 2.2}
By Lemma \ref{lemma:lemma 2.1}  and [\ref{2}, Lemma 3], we have an exact sequence
\begin{equation}
0\longrightarrow \mathcal{O}_{\mathbb{P}^2} \longrightarrow E(-1) \longrightarrow \mathcal{I}_{\{ P \}} \longrightarrow 0. 
\end{equation}
Tensoring it by $\mathcal{O}_{\mathbb{P}^2}(1)$ and taking the cohomology we see that 

\begin{equation*}
     H^{0}(\mathbb{P}^2, E) \cong H^{0}( \mathbb{P}^2, \mathcal{O}_{\mathbb{P}^2}(1)) \Oplus H^{0}( \mathbb{P}^2 , \mathcal{I}_{P}(1) ).
\end{equation*}
\\
This shows that dim $H^{0}(\mathbb{P}^2,E)=5$.
\end{remark}

Let $\mathbb{P}(E)$ be the projectivization of $E$ and  $$ \pi:\mathbb{P}(E) \longrightarrow \mathbb{P}^2$$ be the natural projection and  $\mathcal{O}_{\mathbb{P}(E)}(1)$ be the tautological bundle over $\mathbb{P}(E)$. Since $$ H^{0}(\mathbb{P}^2, E)\cong H^{0} (\mathbb{P}(E), \mathcal{O}_{\mathbb{P}(E)}(1)) $$ so we conclude that dim $H^{0}(\mathbb{P}(E),\mathcal{O}_{\mathbb{P}(E)}(1)) =5$  and $\mathcal{O}_{\mathbb{P}(E)}(1)$ is globally generated as E is globally generated, thus linear system defined by  $\mathcal{O}_{\mathbb{P}(E)}(1)$ is base point free. Hence we get a non-degenerate mapping $$\psi :\mathbb{P}(E) \longrightarrow \mathbb{P}^ 4.$$ 
\begin{theorem}\label{theorem:theorem 2.3}
The image of $\psi$ is a quadric $S$ in $\mathbb{P}^4$.
\end{theorem}
\begin{proof}
Let $ \xi=\mathcal{O}_{\mathbb{P}(E)}(1) $ be the tautological bundle over $\mathbb{P}(E)$. By [\ref{7}, Chapter 4, Section 4], we have
\begin{equation}
\xi^{3} = \textrm{deg}(\psi). \textrm{deg(Img }\psi)
\end{equation}  
where Img $\psi$ denotes the image of the morphism $\psi$. 
As the rank of the vector bundle is two, so in the cohomology ring of $\mathbb{P}(E)$ the relation $$\xi^{2}= \pi^{\ast}(c_{1}(E) )\xi- \pi^{\ast}(c_{2}(E))$$ holds, where $c_{i}$ are the Chern classes of vector bundle $E$, for $i\in\{1,2\}$. From the above equation we have $$ \xi^{3}= \pi^{\ast}(c_{1}(E)) \xi^{2}- \pi^{\ast}(c_{2}(E)) \xi.$$  Using this equation it is easy to see that $$ \xi^{3} = 2 = \textrm{deg}(\psi). \textrm{deg(Img} \medspace \psi)$$  as $\psi$ is non-degenerate map  we have deg(Img $\psi)=2$, i.e., image of $\psi$ is a quardic in $\mathbb{P}^{4}$. 
\end{proof}
Our next aim is to show that $S$ is a smooth quadric. To show this, we will prove a lemma and a theorem.

\begin{lemma}\label{lemma:lemma 2.4}
Let \lq$s$\rq \medspace be the  unique section of $E(-1)$ and \lq$\tilde{s}$\rq \medspace be the corresponding section of $\mathcal{O}_{\mathbb{P}(E(-1))}(1)$, $(\tilde{s})_{0} \subseteq \mathbb{P}(E(-1)) \cong \mathbb{P}(E)$  then $\psi $  maps $(\Tilde{s})_{0}$ to a line $\ell$ in $\mathbb{P}^{4}.$ 
\end{lemma}
\begin{proof}
By Lemma \ref{lemma:lemma 2.1} we have  $(s)_{0}= \{P\} $ with $P\in \mathbb{P}^{2}$. Since $$ H^{0}(\mathbb{P}^{2}, E(-1))\cong H^{0}(\mathbb{P}(E(-1)), \mathcal{O}_{\mathbb{P}(E(-1))}(1)) .$$  Let \lq $\tilde{s}$\rq \medspace be the corresponding section to \lq
$s$\rq \medspace of $\mathcal{O}_{\mathbb{P}(E(-1))}(1)$.  By [\ref {6}, Lemma 2.5], $(\Tilde{s})_{0}$ is the blow-up of $\mathbb{P}^2 $ along the point $P$ i.e.,  $$(\tilde{s})_{0}=\widetilde{\mathbb{P}^{2}_{P} }\xrightarrow[]{\theta} \mathbb{P}^2.$$  Since we know that $$\widetilde{\mathbb{P}^{2}_{P}} \cong \mathbb{P}(\mathcal{O}_{\mathbb{P}^{1}} \Oplus \mathcal{O}_{\mathbb{P}^{1}}(1)) \xrightarrow[]{\beta} \mathbb{P}^{1}$$  is projective bundle over $\mathbb{P}^{1}$. To show the claim it is enough to show that $$\mathcal{O}_{\mathbb{P}(E)}(1) \mid_{(\Tilde{s})_{0}}\cong \beta^{\ast}(\mathcal{O}_{\mathbb{P}^{1}}(1)).$$ By projection formula we have  $$R^{i}\theta_{\ast}(\theta^{\ast}(\mathcal{O}_{\mathbb{P}^2}(1))\Otimes \mathcal{O}_{\widetilde{\mathbb{P}^{2}_{P}}}(-D))\cong R^{i}\theta_{\ast}(\mathcal{O}_{\widetilde{\mathbb{P}^{2}_{P}}}(-D)) \Otimes \mathcal{O}_{\mathbb{P}^2}(1)$$ where $D$ is the exceptional curve. Since $R^{i}\theta_{\ast}(\mathcal{O}_{\widetilde{\mathbb{P}^{2}_{P}}}(-D)) =  0$, for all $i>0$ so by [\ref{9}, Chapter III.9, Exercise 8.1], we have  $$ H^{0}(\medspace \widetilde{\mathbb{P}^{2}_{P}}, \theta^{\ast}(\mathcal{O}_{\mathbb{P}^2}(1))\Otimes \mathcal{O}_{\widetilde{\mathbb{P}^{2}_{P}}}(-D))\cong H^{0}(\mathbb{P}^2, \mathcal{I}_{P}(1))$$ where $\mathcal{I}_{P}$ is the ideal sheaf of $P\in \mathbb{P}^2$. As dim $H^{0}(\mathbb{P}^2, \mathcal{I}_{P}(1)) = 2 $,  we have the following diagram

\begin{center}
\begin{tikzcd}

\widetilde{\mathbb{P}^2_{P}}\cong\mathbb{P}(\mathcal{O}_{\mathbb{P}^1}\Oplus\mathcal{O}_{\mathbb{P}^1}(1))\arrow[rd, "\beta"] \arrow[d,"\theta"]
 \\
\mathbb{P}^2 \arrow[r, dashed, ]
& |[, rotate=0]|  \mathbb{P}^1
\end{tikzcd}.
\end{center}
and
$$\theta^{\ast}(\mathcal{O}_{\mathbb{P}^2}(1))\Otimes \mathcal{O}_{\widetilde{\mathbb{P}^{2}_{P}}}(-D) \cong \beta^{\ast}\mathcal{O}_{\mathbb{P}^{1}}(1).$$ By properties of blow-up there is a  surjective map $$ \theta^{\ast}(\mathcal{I}_{P}) \rightarrow \mathcal{O}_{\widetilde{\mathbb{P}^{2}_{P}}}(-D).$$ It will induce a surjective map between $$\theta^{\ast}(\mathcal{I}_{P}(1)) \rightarrow \theta^{\ast}(\mathcal{O}_{\mathbb{P}^2}(1)) \Otimes \mathcal{O}_{\widetilde{\mathbb{P}^{2}_{P}}}(-D).$$ By the exact sequence (1), we have a surjective map $$ \theta^{\ast}(E)\rightarrow \theta^{\ast}(\mathcal{O}_{\mathbb{P}^2}(1)) \Otimes \mathcal{O}_{\widetilde{\mathbb{P}^{2}_{P}}}(-D) .$$ 
By [\ref{9}, Chapter II.7, Proposition 7.12 ], we have the following commutative diagram
\begin{center}
\begin{tikzcd}
\widetilde{\mathbb{P}^2_{P}} \arrow[r, hook]  \arrow[d, "\theta"]
& \mathbb{P}(E) \arrow[ld, "\pi"] \\
\mathbb{P}^{2}
\end{tikzcd}
\end{center}

and
$$\mathcal{O}_{\mathbb{P}(E)}(1) \mid_{(\Tilde{s})_{0}}\cong \beta^{\ast}(\mathcal{O}_{\mathbb{P}^{1}}(1)).$$
Using this, we conclude the lemma.
\end{proof}
\begin{remark}\label{remark:remark 2.5}
The line $ \ell$  is also contained in quadric $ S$.
    
\end{remark}
Now we are ready to prove the Theorem \ref{foo}
\begin{proof}[\textbf{Proof of Theorem \ref{foo}}]
 As $E$ is globally generated, using equation (1), we have a surjective map $$ \mathcal{O}_{\mathbb{P}^2}(1)\Oplus \mathcal{O}_{\mathbb{P}^2}^{\bigoplus 2} \longrightarrow E$$, i.e., we have the following exact sequence 
 \begin{equation}
 0\longrightarrow \mathcal{G} \longrightarrow \mathcal{O}_{\mathbb{P}^2}(1)\Oplus \mathcal{O}_{\mathbb{P}^2}^{\bigoplus 2} \longrightarrow E \longrightarrow 0 
 \end{equation}
 where $\mathcal{G}$ is the kernel of the map  $ \mathcal{O}_{\mathbb{P}^2}(1)\Oplus
 \mathcal{O}_{\mathbb{P}^2}^{\bigoplus 2} \longrightarrow E$. By the Chern class computation we have  $\mathcal{G} \cong \mathcal{O}_{\mathbb{P}^2}(-1)$ i.e., an exact sequence 
$$ 0\longrightarrow \mathcal{O}_{\mathbb{P}^2}(-1) \longrightarrow \mathcal{O}_{\mathbb{P}^2}(1)\Oplus \mathcal{O}_{\mathbb{P}^2}^{\bigoplus 2} \longrightarrow E \longrightarrow 0 .$$ If $M = \mathcal{O}_{\mathbb{P}^2}(1)\Oplus \mathcal{O}_{\mathbb{P}^2}^{\bigoplus 2}$ then $M$  is globally generated  and dim $ H^{0}(\mathbb{P}^2,M)= 5$. Let $\mathbb{P}(M)$ be the projectivization  of $M$, $\mathcal{O}_{\mathbb{P}(M)}(1)$ be the tautological bundle over $\mathbb{P}(M)$. The linear system defined by $\mathcal{O}_{\mathbb{P}(M)}(1)$ is base point free and we have non-degenerate mapping $$\eta : \mathbb{P}(M)  \longrightarrow \mathbb{P}^4  .$$ After twisting $M$ by $\mathcal{O}_{\mathbb{P}^2}(-1)$ we have $ M(-1) \cong \mathcal{O}_{\mathbb{P}^2}\Oplus \mathcal{O}_{\mathbb{P}^2}^{\bigoplus 2}(-1)$, let $v\in H^{0}(\mathbb{P}^2,M(-1)) $ be a unique section of $M(-1)$, \lq$\Tilde{v} $\rq  \medspace be the corresponding section of $\mathcal{O}_{\mathbb{P}(M(-1))}(1)$ and $(\Tilde{v})_{0}= \mathbb{P}(\mathcal{O}_{\mathbb{P}^2}^{\bigoplus 2}(-1) )=\mathbb{P}^2 \times \mathbb{P}^1 \subseteq \mathbb{P}(M(-1))\cong \mathbb{P}(M)$ and $\eta $ maps   $(\Tilde{v})_{0}$ to a line in $\mathbb{P}^4$ as $ \mathcal{O}_{\mathbb{P}(M(-1))}(1)\mid_{\mathbb{P}(\mathcal{O}_{\mathbb{P}^2}^{\bigoplus 2}(-1) )} $ has two independent sections. After twisting the above exact sequence  by $\mathcal{O}_{\mathbb{P}^2}(-1)$  we have the following exact sequence 
$$  0\longrightarrow \mathcal{O}_{\mathbb{P}^2}(-2) \longrightarrow \mathcal{O}_{\mathbb{P}^2}\Oplus \mathcal{O}_{\mathbb{P}^2}^{\bigoplus 2} (-1) \longrightarrow E(-1) \longrightarrow 0 $$   and $\mathbb{P}(E(-1)) \hookrightarrow
\mathbb{P}(M(-1)) $ and $\mathcal{O}_{\mathbb{P}(M(-1))}(1) \mid_{\mathbb{P}(E(-1))} =\mathcal{O}_{\mathbb{P}(E(-1))}(1) $
and unique section \lq$\Tilde{v}$\rq \medspace of $ \mathcal{O}_{\mathbb{P}(M(-1))}(1)$ corresponds to a unique section \lq$\Tilde{s}$\rq \medspace of $\mathcal{O}_{\mathbb{P}(E(-1))}(1)$ and we have $$ (\Tilde{v})_{0} \cap \mathbb{P}(E(-1))= (\Tilde{s})_{0}.$$ Using the following commutative diagram

\begin{center}
\begin{tikzcd}
\mathbb{P}(E(-1))\cong \mathbb{P}(E) \arrow[r, hook   ] \arrow[rd,"\psi"]
& \mathbb{P}(M(-1))\cong \mathbb{P}(M) \arrow[d, "\eta" ] \\
& |[, rotate=0]|  \mathbb{P}^4
\end{tikzcd}
\end{center}
\noindent
we see that $\eta$ maps $(\Tilde{v})_{0} $ to the same line $ \psi((\Tilde{s})_{0})=\ell$ in $\mathbb{P}^4$. As $ \ell\subseteq \mathbb{P}^4$ is a linear variety so by  [\ref{5}, Chapter 9,  Proposition 9.11] we have an isomorphism  $$  \widetilde{\mathbb{P}^4_{\ell}} \cong \mathbb{P}(M) \longrightarrow \mathbb{P}^2$$ and under this isomorphism the blow-up map  $\widetilde{\mathbb{P}^4_{\ell}} \longrightarrow \mathbb{P}^4 $ given by the complete linear system defined by  $\mathcal{O}_{\mathbb{P}(M)}(1)$. So we have the following commutative diagram,

\begin{center}
\begin{tikzcd}
\mathbb{P}(E) \arrow[r, hook,   ] \arrow[d,"\psi"]
& \mathbb{P}(M)\cong \widetilde{\mathbb{P}^4_{\ell}} \arrow[d, "\eta" ] \\
S \arrow[r, hook, ]
& |[, rotate=0]|  \mathbb{P}^4
\end{tikzcd}
\end{center}

and we know $\psi$ is a degree one surjective morphism, so it is a birational morphism. By [\ref{9} Chapter II.7, Theorem 7.17 ]  $\psi$ is a blowing-up map. Using the above diagram and $ \ell \subseteq S $, we conclude that $\mathbb{P}(E) \cong \widetilde{S_{\ell}} $.
\end{proof}
   
\begin{lemma}\label{lemma:lemma 2.7}
Let $ X$ be an irreducible quadric in $\mathbb{P}^4$, and $ X$ is singular along a point. Let $ \ell$ be a line passing through that point, then $ \widetilde{X}_{\ell}$ the blow-up of $X$ along $\ell$ is singular.

\begin{proof}
Let $X$ be an irreducible quadric, which is singular along a point; then, after a change of coordinates, we can assume $X = Z(X_{0}X_{1}-X_{2}X_{3}) \subseteq \mathbb{P}^4$ and the singular point will be  $  [\ 0:0:0:0:1]\ $ and $ \ell $ is a line defined by $X_{0}=0, X_{1}=0, X_{2}=0 $.  Now blow-up of $ \mathbb{P}^4$  along the line $ \ell$ is $$\widetilde{\mathbb{P}^4_{\ell}} =\{  [ X_{0}:X_{1}:X_{2}:X_{3}:X_{4}] \times [ Y_{0}:Y_{1}:Y_{2} ] \in \mathbb{P}^4 \times \mathbb{P}^2  : X_{0}Y_{1}-X_{1}Y_{0}=0,$$  $$X_{0}Y_{2}-X_{2}Y_{0}=0, X_{1}Y_{2}-X_{2}Y_{1}=0   \}$$ 
where $ X_{0}, X_{1},X_{2},X_{3},X_{4}$ are homogeneous coordinates of $\mathbb{P}^4$ and $ Y_{0},Y_{1},Y_{2}$ homogeneous coordinates of $ \mathbb{P}^2$.
Suppose $ \omega : \widetilde{\mathbb{P}^4_{\ell}} \longrightarrow \mathbb{P}^4$ be the blowing-up  then the total transform $ \omega^{-1}(X)  $ defined by the equations $ X_{0}Y_{1}-X_{1}Y_{0}=0,  X_{0}Y_{2}-X_{2}Y_{0}=0, X_{1}Y_{2}-X_{2}Y_{1}=0, X_{0}X_{1}-X_{2}X_{3}=0 $ and the exceptional divisor $$ \omega^{-1}(\ell) \cong \mathbb{P}^1 \times \mathbb{P}^2 . $$ Consider the open set given by $ X_{4}=1 $ and $ Y_{0}=1$ in $ \mathbb{P}^4 \times \mathbb{P}^2$.  We can identify this open set with $ \mathbb{A}^4 \times \mathbb{A}^2 = \mathbb{A}^6$ with affine  coordinates $X_{0},X_{1},X_{2},X_{3}, Y_{1},Y_{2}$. We have intersection $$\widetilde{\mathbb{P}^4_{\ell}} \cap \mathbb{A}^6 = Z( X_{0}Y_{1}-X_{1}, \medspace X_{0}Y_{2}-X_{2}) $$  and we see that  $$ (Z( X_{0}Y_{1}-X_{1}, \medspace X_{0}Y_{2}-X_{2}) \subseteq \mathbb{A}^6)  \cong \mathbb{A}^4 $$  with affine coordinates $ X_{0}, X_{3}, Y_{1}, Y_{2}$. In $\mathbb{A}^4$ part of  strict transform of X i.e., $ \widetilde{X_{\ell}}\cap \mathbb{A}^4$ defined by $ Z (Y_{1}X_{0}-Y_{2}X_{3}) $ which is  singular  at $ (0,0,0,0)$. So $ \widetilde{X_{\ell}}$ is not smooth, as smoothness is a local property.   
          
    \end{proof}
\end{lemma}

\begin{corollary}\label{corollary:corollary 2.8}
The quadric S is smooth in $\mathbb{P}^4$.
\end{corollary}
\begin{proof}
From Theorem \ref{foo} we see that $\mathbb{P}(E)$ is a smooth subvariety of $\widetilde {\mathbb{P}^4_{\ell}}$
and also $ \mathbb{P}(E)-\psi^{-1}(\ell) \cong S- \ell $ hence $ S- \ell$ is smooth. If $ S$  is not smooth, it has singularity along the whole $ \ell$ or at some point of $ \ell$.\\\\
$ \underline{Case} (1) $: Suppose $S$ is singular along $\ell$. Consider the projection $$ \mathbb{P}^4 - \ell  \xlongrightarrow{p} \mathbb{P}^2$$  and  $$ p _{\mid_{S-\ell}} : S-\ell \longrightarrow  \mathbb{P}^2 $$ be the restriction of this projection to $ S-\ell $. The image of $  p _{\mid_{S-\ell}} $ is a smooth conic  $C$ in $ \mathbb{P}^2$ as $S-\ell$ is smooth. $S$ is a cone over  conic $C$ with vertex along $\ell$
 and by information given on rational normal scrolls in [\ref{1}, page number 95, 96] we see that  
 $ \widetilde{S_{\ell}} \cong \mathbb{P}(\mathcal{O}_{\mathbb{P}^1}(2)\Oplus \mathcal{O}_{\mathbb{P}^1} \Oplus \mathcal{O}_{\mathbb{P}^1})$. Projection $ p$ gives the  morphism $ p_{1} : \widetilde {\mathbb{P}^4_{\ell}} \longrightarrow \mathbb{P}^2 $ and $p_{1} $ is the natural projection $ \mathbb{P}(M)\longrightarrow \mathbb{P}^2$ as $ \mathbb{P}(M) \cong \widetilde {\mathbb{P}^4_{\ell}} $. Now 
by Theorem \ref{foo} we get $\widetilde{S_{\ell}} \cong \mathbb{P}(E) $. So we have $ \mathbb{P}(E) \cong \mathbb{P}(\mathcal{O}_{\mathbb{P}^1}(2)\Oplus \mathcal{O}_{\mathbb{P}^1} \Oplus \mathcal{O}_{\mathbb{P}^1})$. But  this is not possible as by [\ref{3}, Theorem 4.2  ] Chow rings  $$ A^{\ast}(\mathbb{P}(E)) \ncong A^{\ast}(\mathbb{P}(\mathcal{O}_{\mathbb{P}^1}(2)\Oplus \mathcal{O}_{\mathbb{P}^1} \Oplus \mathcal{O}_{\mathbb{P}^1})). $$  So $S$ can not be singular along $ \ell$.\\
\\
$ \underline{Case} $ (2): By Lemma \ref{lemma:lemma 2.7},  $ S$ can not be singular along a point on $ \ell $ as $ \mathbb{P}(E)\cong \widetilde{S_{\ell}}$ is smooth. 
\end{proof}

\begin{corollary}\label{corollary:Corollary 2.9}
The vector bundle $ E $ is nef and big but not ample.
\end{corollary}
\begin{proof}
As E is globally generated, hence it is nef. From Theorem \ref{foo} we see that the map given by the linear system $| \mathcal{O}_{\mathbb{P}(E)}(1) |$ is a blow-up map and it maps exceptional divisor $ (\tilde{s})_{0}= \widetilde{\mathbb{P}^2_{P}}$ to a line. So, it can not be a finite map, so E is not an ample vector bundle. As the map $\psi$ defined by $ \mathcal{O}_{\mathbb{P}(E)}(1)$ is birational onto its image so $E$  is big. 
 
\end{proof}
\begin{remark}
The vector bundle, $E$, is semi-stable and big but not ample over $ \mathbb{P}^2$, but a semi-stable big vector bundle is ample over a smooth projective curve. 
\end{remark}
\begin{remark}
The converse of the [\ref{10}, Theorem 3 ] is not true as the Chern classes of vector bundle $E$ satisfy the relation $c_{2}\gtr 0,c_{1}^{2}-c_{2}\gtr 0 $ but $E$ is not an ample bundle.
\end{remark}
\begin{remark}
  The projective variety $\widetilde{S_{\ell}}  $ is a fano variety. Since, by Theorem \ref{foo} we have $ \widetilde{S_{\ell}}\cong \mathbb{P}(E)$ and it is easy to see that anti-canonical bundle of $\mathbb{P}(E)$ i.e.$$ \Check{K_{\mathbb{P}(E)}}= \mathcal{O}_{\mathbb{P}(E)}(2) \bigotimes \pi^{\ast} \mathcal{O}_{\mathbb{P}^2}(1) $$ is ample bundle.
    \end{remark}

\section{Correspondence between Projective bundles over \texorpdfstring{$ \mathbb{P}^2 $}{Text} and smooth quadrics in \texorpdfstring{$\mathbb{P}^4$}{Text}}

Now consider the following morphism of locally free sheaves  $$ \mathcal{O}_{\mathbb{P}^2}\xrightarrow{\alpha} \mathcal{O}_{\mathbb{P}^2}(2)\Oplus\mathcal{O}_{\mathbb{P}^2}^{\bigoplus 2}(1)$$
 $$ 1\mapsto (f,g,h)$$
where $f\in H^{0}(\mathbb{P}^2,\mathcal{O}_{\mathbb{P}^2}(2))$ and $ g,h \in H^{0}(\mathbb{P}^2, \mathcal{O}_{\mathbb{P}^2}(1))$ such that $Z(f,g,h)= \varnothing$. Then we have the following exact sequence 
\begin{equation}
0\longrightarrow \mathcal{O}_{\mathbb{P}^2}\xlongrightarrow{\alpha} \mathcal{O}_{\mathbb{P}^2}(2)\Oplus \mathcal{O}_{\mathbb{P}^2}^{\bigoplus 2}(1) \longrightarrow N \longrightarrow 0 
\end{equation}
where  $ N$ is cokernel, and it's a rank two vector bundle. Tensoring the exact sequence  by $\mathcal{O}_{\mathbb{P}^2}(-1)$ we  get the following sequence  
\begin{equation}
0\longrightarrow \mathcal{O}_{\mathbb{P}^2}(-1)\xlongrightarrow{\alpha} \mathcal{O}_{\mathbb{P}^2}(1)\Oplus \mathcal{O}_{\mathbb{P}^2}^{\bigoplus 2} \longrightarrow V \longrightarrow 0 
\end{equation}
where $V \cong 
N(-1) $ a rank two vector bundle over $\mathbb{P}^2$.
Let $\mathbb{P}(V)$ be the projectivization of $V$ and  $$ \tau:\mathbb{P}(V) \longrightarrow \mathbb{P}^2$$ be the natural projection 
and  $\mathcal{O}_{\mathbb{P}(V)}(1)$ be the tautological bundle over $\mathbb{P}(V)$. Since $$ H^{0}(\mathbb{P}^2, V)\cong H^{0} (\mathbb{P}(V), \mathcal{O}_{\mathbb{P}(V)}(1)) $$  we conclude that dim $H^{0}(\mathbb{P}(V),\mathcal{O}_{\mathbb{P}(V)}(1)) =5$  and $\mathcal{O}_{\mathbb{P}(V)}(1)$ is globally generated as $V$ is globally generated. So the linear system defined by  $\mathcal{O}_{\mathbb{P}(V)}(1)$ is base point free. Thus we have a non-degenerate mapping $$\rho:\mathbb{P}(V) \longrightarrow \mathbb{P}^ 4.$$ 

\begin{theorem}\label{theorem:theorem 3.1}
The image of $ \rho $, Img $(\rho)$ is a smooth quadric $ F_{\alpha} $ in $ \mathbb{P}^4$. 
\end{theorem}
\begin{proof}
Suppose $H^{0}(\mathbb{P}^2, V)= H^{0}(\mathbb{P}(V), \mathcal{O}_{\mathbb{P}(V)}(1))= \langle X_{0}, X_{1}, X_{2},U,W \rangle$.
 We have  surjective morphism  $$  \textrm{Sym}^{2}H^{0}(\mathbb{P}(V),\mathcal{O}_{\mathbb{P}(V)}(1))\xlongrightarrow{\Psi} H^{0}(\mathbb{P}(V),\mathcal{O}_{\mathbb{P}(V)}(2)) .$$ We know that  $$H^{0}(\mathbb{P}^2, \textrm{Sym}^{2}V)= H^{0}(\mathbb{P}(V),\mathcal{O}_{\mathbb{P}(V)}(2)) = H^{0}(\mathbb{P}(V),\rho^{\ast}(\mathcal{O}_{\mathbb{P}^4}(2)))$$ and in exact sequence  (4) $\alpha$ is defined by $1\rightarrow (f,g,h)$ so we see that $ f + Ug+ Wh \in$ Ker$ (\Psi)$ i.e., $ f + Ug+ Wh =0$ along the Img $(\rho )$. So we have Img $(\rho) \subseteq  Z(f + Ug+ Wh) $. By Theorem \ref{theorem:theorem 2.3} Img $(\rho)$ is codimension  one subvariety of $\mathbb{P}^4$. Using the Jacobian criterion, we see that $ Z(f+Ug+Wh)$ is a smooth quadric and hence irreducible. So by the dimension and irreducibility of Img $(\rho) $ and $Z(f + Ug+ Wh)$   we  conclude that Img $(\rho) = F_{\alpha} = Z(f + Ug+ Wh)$. So, $F_{\alpha}$ is a smooth quadric in $ \mathbb{P}^4$. 
    
\end{proof}
\begin{remark}
From (3), we see that $E$ fits into the exact sequence of the form (5), then Theorem \ref{theorem:theorem 3.1} gives an alternate proof of Corollary \ref{corollary:corollary 2.8}.

\end{remark}

\begin{example}

In the above theorem if we choose  $f= X_{0}^2$, $ g=X_{1}$ and $h= X_{2} $ then the image of $ \rho$ is $ X_{0}^{2} + UX_{1}+ WX_{2}$ which is clearly a smooth quadric in $ \mathbb{P}^4$.
    
\end{example}

If we choose $g$, $h$ such that  $Z(g,h) = \{{P}\}$ where $P\in \mathbb{P}^2$ is such that $(s)_{0}=P$, $s\in H^{0}(\mathbb{P}^2, E(-1))$ is the unique section then we have the following proposition.
\begin{proposition}\label{propostion:proposition 3.4}
 Let $V$ be as above, then $V$ is semistable but not stable, and $V(-1) $ has a unique section \lq $t$\rq \medspace such that $(t)_{0}$ is a single point $P \in \mathbb{P}^2$ and $V \cong E $.
\end{proposition}
\begin{proof}
Using the above exact sequence (5) we have $$c_{1}(V)= 2[h]\quad,\quad c_{2}(V)=2[h^2] $$ 
and after twisting the exact sequence (5) by $\mathcal{O}_{\mathbb{P}^2}(-2) $ we have
\begin{equation}
0\longrightarrow \mathcal{O}_{\mathbb{P}^2}(-1)\xlongrightarrow{\alpha} \mathcal{O}_{\mathbb{P}^2}\Oplus \mathcal{O}_{\mathbb{P}^2}^{\bigoplus 2}(-1) \longrightarrow V(-1) \longrightarrow 0 .
\end{equation}
Using the above exact sequence, we conclude that

$$c_{1}(V(-1))= 0 \quad,\quad c_{2}(V(-1))=[h^2] $$

and $H^{0}(\mathbb{P}^2, V(-1))= 1$. So by [\ref{11} Chapter 2, Lemma 1.2.5], $V(-1)$  can not be a stable bundle as  $V(-1)$ is a normalized bundle and also since  $H^{0}(\mathbb{P}^2, V(-2))= 0$ so  $V(-1)$ is semi-stable. Thus, we see that $V$ is not stable but semi-stable. Let \lq$t$\rq \medspace be the unique section of $V(-1) $. As $ \alpha$ defined by $ 1\rightarrow (f,g,h) $ and $ Z(g,h)= P $ so by exact sequence (6), we see that $ P\in (t)_{0} $ and by  Lemma \ref{lemma:lemma 2.1}, we have $(t)_{0}= P $. Using  Remark \ref{remark:remark 2.2}, we have the following exact sequence  $$ 0\longrightarrow \mathcal{O}_{\mathbb{P}^2} \longrightarrow V(-1) \longrightarrow \mathcal{I}_{\{ P \}} \longrightarrow 0 .$$  Since   $ Ext^{1}(\mathcal{I}_{P},\mathcal{O}_{\mathbb{P}^2})\cong \mathbb{C} $ from this we conclude  $V(-1)\cong E(-1) $ i.e., $ V \cong E$.
\end{proof}

\begin{proof}[\textbf{Proof of Theorem \ref{baz}}] 
Let $M^{'} = \mathcal{O}_{\mathbb{P}^2}(2)\Oplus \mathcal{O}_{\mathbb{P}^2}^{\bigoplus 2}(1)$ be the vector bundle and  $ W=H^{0}(M^{'})$.  Firstly, we will show the existence of an open set $U$. We have a surjective morphism of locally free sheaves $$W \Otimes \mathcal{O}_{\mathbb{P}^2} \xlongrightarrow{\varepsilon } M^{'}. $$ Then we have the following exact sequence

$$ 0\longrightarrow K \longrightarrow W \Otimes \mathcal{O}_{\mathbb{P}^2} \xlongrightarrow{\varepsilon } 
M^{'}\longrightarrow 0 $$
where $K$ is the kernel of $\varepsilon $. From the above exact sequence, we have $ \mathbb{P}(K) \subseteq \mathbb{P}( W \Otimes M^{'})= \mathbb{P}^2 \times \mathbb{P}(W)$ and $$ \mathbb{P}(K) =\{ (x,s)\in \mathbb{P}^2\times \mathbb{P}(W): s_{x}\in m_{x}M^{'}_{x}  \} .$$ Let $ \varepsilon_{x}: W\longrightarrow M^{'}_{x}\Otimes k(x)$ define by $ s\longmapsto \Bar{s}_{x}$ then we see that $$ K_{x}= \textrm{Ker}\medspace \varepsilon_{x} = \{s\in W : s_{x}\in m_{x}M^{'}_{x} \} $$ 
$$ = \{ s\in W : \Bar{s}_{x}=0 \medspace \mathrm{in} \medspace M^{'}\Otimes k(x) \} .$$

Clearly map $\varepsilon_{x}$ is surjective as $W$ generates $M^{'}$. Let rk$ (M^{'}_{x})$ denotes the rank of the module  $M^{'}_{x}$ so we have   dim (Ker $\varepsilon_{x})$=  dim $W-\textrm{rk} \medspace (M^{'}_{x})$ i.e., dim (Ker $\varepsilon_{x})$=  dim $W-3$. From this, we  conclude that dim $\mathbb{P}(K_{x})\leq $ dim $\mathbb{P}(W) -3 $. Also
we observe that the fiber of $ \mathbb{P}(K)$ over $ x\in \mathbb{P}^2$ is $  \mathbb{P}(K)_{x}= \mathbb{P}(K_{x}) $. So from above discussion we see that dim $ \mathbb{P}(K)\leq \textrm{dim}\medspace \mathbb{P}(W) -3 + \textrm  {dim}\medspace(\mathbb{P}^2)$  i.e., dim $ \mathbb{P}(K)\less \textrm{dim} \medspace \mathbb{P}(W)$ as dim $\mathbb{P}^2  \less 3  $. Consider the following diagram  
\begin{center}
\begin{tikzcd}
\mathbb{P}^2\times \mathbb{P}(W) \arrow[r, "p_{2}"]  \arrow[d, "p_{1}"]
& \mathbb{P}(W)\\
\mathbb{P}^2
\end{tikzcd}
\end{center}
where  $p_{1}$ and $ p_{2}$  be the projections.
 As $\mathbb{P}^2$ is projective, $p_{2}$ is proper so $p_{2}(\mathbb{P}(K))$ is closed subset of $\mathbb{P}(W)$ and it is proper closed subset of $ \mathbb{P}(W)$  as  we know that dim $p_{2}(\mathbb{P}(K))\less $ dim $\mathbb{P}(W)$. Take $ U= \mathbb{P}(W)-p_{2}(\mathbb{P}(K))$ clearly, $U$ is non-empty open subset of $\mathbb{P}(W)$. For each $ s\in U$  we see that \lq $s$\rq \medspace is a nowehre vanishng section of $M^{'}$, then cokernel of the following map $$ \mathcal{O}_{\mathbb{P}^2} \longrightarrow M^{'}$$ $$1\longrightarrow s $$
is a vector bundle denoted by $V^{'}_{s}$ over $ \mathbb{P}^2$. Let $V_{s}=V^{'}_{s}(-1)$  then second statement of the theorem is followed by  Remark \ref{remark:remark 2.5}, Theorem \ref{foo} and Corollary \ref{corollary:corollary 2.8}.
\end{proof} 
\begin{remark}
We observe that we get a family of not-ample, nef, and big vector bundles over $\mathbb{P}^2$ parameterized by an open set of $ \mathbb{P}(H^{0}(M^{'}))$ whose projectivization corresponds to a family of smooth quadrics containing a fixed-line in $ \mathbb{P}^4$.
\end{remark}

\section{Correspondence between Projective bundles over  \texorpdfstring{$\mathbb{P}^2 $}{Text} and rational hypersurfaces of degree \texorpdfstring{$n(\geq 3)$}{Text} in \texorpdfstring{$\mathbb{P}^4$}{Text}}
 
 In general, for $n\geqslant 3$ consider the following morphism of locally free sheaves  $$ \mathcal{O}_{\mathbb{P}^2}\xrightarrow{\gamma} \mathcal{O}_{\mathbb{P}^2}(n)\Oplus\mathcal{O}_{\mathbb{P}^2}^{\bigoplus 2}(n-1)$$
 $$ 1\mapsto (f_{n},g_{n},h_{n})$$
where $f_{n}\in H^{0}(\mathbb{P}^2,\mathcal{O}_{\mathbb{P}^2}(n))$ and $ g_{n},h_{n} \in H^{0}(\mathbb{P}^2, \mathcal{O}_{\mathbb{P}^2}(n-1))$ such that $Z(f_{n},g_{n},h_{n})= \varnothing$. Then we have the following exact sequence
\begin{equation}
0\longrightarrow \mathcal{O}_{\mathbb{P}^2}\xlongrightarrow{\gamma} \mathcal{O}_{\mathbb{P}^2}(n)\Oplus \mathcal{O}_{\mathbb{P}^2}^{\bigoplus 2}(n-1) \longrightarrow N_{n} \longrightarrow 0 
\end{equation} 
where $ N_{n}$ is the cokernel. After  twisting (7) by $\mathcal{O}_{\mathbb{P}^2}(1-n)$ we get the following   
\begin{equation}
0\longrightarrow \mathcal{O}_{\mathbb{P}^2}(1-n)\longrightarrow \mathcal{O}_{\mathbb{P}^2}(1)\Oplus \mathcal{O}_{\mathbb{P}^2}^{\bigoplus 2} \longrightarrow V_{n} \longrightarrow 0
\end{equation}
where $ V_{n}= N_{n}(1-n)$ is a rank two vector bundle over $\mathbb{P}^2$. Let us suppose $ \mathbb{P}(V_{n})$ be the projectivization of $ V_n$ and  $H^{0}(\mathbb{P}^2, V_{n})= H^{0}(\mathbb{P}(V_{n}), \mathcal{O}_{\mathbb{P}(V_{n})}(1))= \langle X_{0},X_{1},X_{2},U,W\rangle$ and  as earlier  we have the non-degenerate mapping $$ \rho_{n}: \mathbb{P}(V_{n}) \longrightarrow \mathbb{P}^4 .$$ 
Let $ F_{n, \gamma }$ be the image of $ \rho_{n}$ . 

\begin{lemma}\label{lemma:lemma 4.1}
Let $V_{n}$ be as above, then $V_{n}$ is stable and $ V_{n}(-1) $ has a unique global section \lq $t_{n}$\rq . If we choose $f_{n}, g_{n}, h_{n}$ such that $Z(f_{n},g_{n},h_{n})= \varnothing$ and $ Z(g_{n},h_{n}) $ is 
 zero dimensional subscheme of length $(n-1)^2$ of $ \mathbb{P}^2$  then the zero locus of \lq $t_{n}$\rq \medspace is equal to $ Z(g_{n},h_{n})$.

\end{lemma}
\begin{proof}
Using (8) we have  $c_{1}(V_{n})= n[h]$ and 
\begin{equation}
(V_{n})_{\textrm{norm}}=
\begin{cases}
V_{n}(-\frac{n}{2}), & \text{if $n$ is even }\ \\
V_{n}(-\frac{n+1}{2})  , & \text{ if $n$ is odd}
\end{cases}.
\end{equation}

If $n$ is even then after twisting (8) by $\mathcal{O}_{\mathbb{P}^2}(\frac{-n}{2})$ and taking the cohomology we have $H^{0}(\mathbb{P}^2, (V_{n})_{\textrm{norm}})= 0$ and similarly if $n$ is odd we have $H^{0}(\mathbb{P}^2, (V_{n})_{\textrm{norm}})= 0$ so by [\ref{11} Chapter 2, Lemma 1.2.5], we conclude that $V_{n}$ is a stable bundle. After twisting  exact sequence (8) by $\mathcal{O}_{\mathbb{P}^2}(-1)$ we have  exact sequence
\begin{equation}
0\longrightarrow \mathcal{O}_{\mathbb{P}^2}(-n)\longrightarrow \mathcal{O}_{\mathbb{P}^2}\Oplus \mathcal{O}_{\mathbb{P}^2}^{\bigoplus 2}(-1) \longrightarrow V_{n}(-1) \longrightarrow 0 .
 \end{equation}

\noindent Now taking the cohomology we  get dim $H^{0}(\mathbb{P}^2, V_{n}(-1))= 1$. Let \lq $ t_{n}$\rq \medspace be the unique section  in  $V_n(-1)$. Since the map $ \gamma$ is defined by $ 1 \rightarrow ( f_{n},g_{n},h_{n}) $ and also we have $Z(f_{n},g_{n},h_{n})= \varnothing$. So, by exact sequence (10), it follows that $ (t_{n})_{0}= Z(g_{n},h_{n})$.

\end{proof}

If we choose $ g_{n},h_{n} \in H^{0}(\mathbb{P}^2, \mathcal{O}_{\mathbb{P}^2}(n-1)) $  such  that $g_{n}$, $ h_{n}$ intersect on $ (n-1)^2$ distinct points. 
Then we have the following lemma. 
 
\begin{lemma}\label{lemma:lemma 4.2}
Let \lq $t_{n}$\rq \medspace be the  unique section of $V_{n}(-1)$ and \lq$\tilde{t_{n}} $\rq \medspace be the corresponding section of $\mathcal{O}_{\mathbb{P}(V_{n}(-1))}(1)$ then $(\tilde{t_{n}})_{0} = \widetilde{\mathbb{P}^2}_{P_{1}P_{2}...P_{r}}  $ where $\widetilde{\mathbb{P}^2}_{P_{1}P_{2}...P_{r}}$ denotes the blow-up of $ \mathbb{P}^2$ at distinct points  $P_{1},P_{2},...,P_{r}$  with $ r=(n-1)^2$ and $\rho_{n} $ maps $ (\tilde{t_{n}})_{0} \subseteq \mathbb{P}(V_{n}(-1)) \cong \mathbb{P}(V_{n})$ to a line $\ell$ in $\mathbb{P}^{4}$. 
\end{lemma}

\begin{proof}
By Lemma \ref{lemma:lemma 4.1} we have  $(t_{n})_{0}= \{P_{1},P_{2},...,P_{r} \}$ a closed subscheme of $ \mathbb{P}^{2}$, where $ P_{1},P_{2},...,P_{r}$ are distinct points with $r= (n-1)^2 $. By [\ref{6}, Lemma 2.5] $(\Tilde{t_{n}})_{0}$ is the blow-up of $\mathbb{P}^2 $ along the points $\{P_{1},P_{2},...,P_{r} \}$  i.e., $$(\tilde{t_{n}})_{0} = \widetilde{\mathbb{P}^2}_{P_{1}P_{2}...P_{r}} \xrightarrow[]{\theta_{n}} \mathbb{P}^2.$$ We have the linear system $W \subseteq H^{0}(\mathbb{P}^2, \mathcal{O}_{\mathbb{P}^2}(n-1))$ such that $ W = \langle g_{n} ,h_{n} \rangle$ where $g_{n},h_{n}\in H^{0}(\mathbb{P}^2, \mathcal{O}_{\mathbb{P}^2}(n-1)) $ and $ Z(g_{n},h_{n})= \{P_{1},P_{2},...,P_{r} \} $ is the base locus of $W$. The linear system $ W $ defines a rational map $$\phi_{n}:\mathbb{P}^2 \xdashrightarrow{\hspace{5mm}} \mathbb{P}^{1} .$$ We have the following diagram 

\begin{center}
\begin{tikzcd}

\widetilde{\mathbb{P}^2}_{{P_{1}P_{2}...P_{r}}}
\arrow[rd, "\beta_{n}"] \arrow[d,"\theta_{n}"]
 \\
\mathbb{P}^2 \arrow[r, dashed, "\phi_{n}" ]
& |[, rotate=0]|  \mathbb{P}^1
\end{tikzcd}
\end{center}

\noindent
and the map $\beta_{n}$ is given by the linear system $$| {\theta_{n}}^{\ast}(\mathcal{O}_{\mathbb{P}^2 }(n-1)) \Otimes \mathcal{O}_{\widetilde{\mathbb{P}^{2}}_{P_{1}P_{2}...P_{r} }}(-D) |$$ where $D$ is the exceptional divisor. So we have $${\theta_{n}}^{\ast}(\mathcal{O}_{\mathbb{P}^2}(n-1))\Otimes \mathcal{O}_{\widetilde{\mathbb{P}^{2}}_{P_{1}P_{2}...P_{r} }}(-D) \cong {\beta_{n}}^{\ast}\mathcal{O}_{\mathbb{P}^{1}}(1).$$

To show the claim it is enough to show $\mathcal{O}_{\mathbb{P}(V_{n})}(1) \mid_{(\Tilde{t_{n}})_{0}}\cong \beta_{n}^{\ast}(\mathcal{O}_{\mathbb{P}^{1}}(1))$. Since the unique section \lq $t_{n} $\rq \medspace  of $V_{n}(-1)$ vanishes only in codimension two then by [\ref{2}, Lemma 3] we have the following exact sequence $$ 0\longrightarrow \mathcal{O}_{\mathbb{P}^2} \longrightarrow V_{n}(-1) \longrightarrow \mathcal{I}_{\{P_{1}, P_{2},..., P_{r} \}}(n-2) \longrightarrow 0 .$$ Then tensoring by $\mathcal{O}_{\mathbb{P}^2}(1)$ we get 
\begin{equation}
0\longrightarrow \mathcal{O}_{\mathbb{P}^2} \longrightarrow V_{n} \longrightarrow \mathcal{I}_{\{P_{1}, P_{2},..., P_{r} \}}(n-1) \longrightarrow 0 .
\end{equation}
By the property of blow-up there is a surjective map $$ {\theta_{n}}^{\ast}(\mathcal{I}_{\{P_{1},P_{2},...,P_{r} \}}) \rightarrow \mathcal{O}_{\widetilde{\mathbb{P}^{2}}_{{P_{1}P_{2}...P_{r} }}}(-D).$$ It will induce the surjective map between $${\theta_{n}}^{\ast}(\mathcal{I}_{\{P_{1},P_{2},...,P_{r} \}}(n-1)) \rightarrow {\theta_{n}}^{\ast}(\mathcal{O}_{\mathbb{P}^2}(n-1)) \Otimes \mathcal{O}_{\widetilde{\mathbb{P}^{2}}_{P_{1}P_{2}...P_{r} }}(-D).$$ By the exact sequence (11), we have the surjective map $$ {\theta_{n}}^{\ast}(V_{n})\rightarrow {\theta_{n}}^{\ast}(\mathcal{O}_{\mathbb{P}^2}(n-1)) \Otimes \mathcal{O}_{\widetilde{\mathbb{P}^{2}}_{P_{1}P_{2}... P_{r} }}(-D) .$$ 
By [\ref{9} Chapter II.7, Proposition 7.12 ], we have the following diagram
\begin{center}
\begin{tikzcd}
\widetilde{\mathbb{P}^2}_{P_{1}P_{2}...P_{r}} \arrow[r, hook ]  \arrow[d, "\theta_{n}"]
& \mathbb{P}(V_{n}) \arrow[ld, "\tau_{n}"] \\
\mathbb{P}^{2}
\end{tikzcd}
\end{center}
and
 $$\mathcal{O}_{\mathbb{P}(V_{n})}(1) \mid_{(\Tilde{t_{n}})_{0}}\cong {\beta_{n}}^{\ast}(\mathcal{O}_{\mathbb{P}^{1}}(1)).$$
Using this, we conclude the lemma.

\end{proof}

\begin{remark}\label{remark:remark 4.3}
    The above line $ \ell$ is also contained in $F_{n,\gamma } $.
\end{remark}
\begin{theorem}\label{theorem:theorem 4.4}The image of $\rho_{n}$ i.e., $F_{n, \gamma}$ is a singular hypersurface of degree $n$.  
\end{theorem}
\begin{proof}
Let $H^{0}(\mathbb{P}^2, V_{n})= H^{0}(\mathbb{P}(V_{n}), \mathcal{O}_{\mathbb{P}(V_{n})}(1))= \langle X_{0}, X_{1}, X_{2},U,W \rangle$.
Since we have the surjective map  $$ \textrm{Sym}^{n}H^{0}(\mathbb{P}(V_{n}),\mathcal{O}_{\mathbb{P}(V_{n})}(1))\xlongrightarrow{\Psi_{n} } H^{0}(\mathbb{P}(V_{n}),\mathcal{O}_{\mathbb{P}(V_{n})}(n)).$$ 
We know that  $$H^{0}(\mathbb{P}^2, \textrm{Sym}^{n}V_{n})= H^{0}(\mathbb{P}(V_{n}),\mathcal{O}_{\mathbb{P}(V_{n})}(n)) = H^{0}(\mathbb{P}(V_{n}),\rho_{n}^{\ast}(\mathcal{O}_{\mathbb{P}^4}(n)))$$ and in the exact sequence  (7) $\gamma$ is defined by $1\rightarrow (f_{n},g_{n},h_{n})$
so we see that $ f_{n} + Ug_{n}+ Wh_{n} \in$ Ker$(\Psi_{n})$ i.e., $ f_{n} + Ug_{n}+ Wh_{n}$ vanishes along the  Img $(\rho_{n})$ so we have Img $(\rho_{n}) \subseteq Z(f_{n}+Ug_{n}+Wh_{n})$. Using the similar argument as we used in Theorem \ref{theorem:theorem 2.3}, we conclude Img $(\rho_{n})$ is codimension one subvariety of $\mathbb{P}^4$. By Jacobian criterion, we see that $ Z(f_{n}+Ug_{n}+Wh_{n})$ is singular along the line  $ X_{0}=0, X_{1}=0, X_{2}=0$ hence $Z(f_{n}+Ug_{n}+Wh_{n})$ is irreducible. So by the dimension and irreducibility of Img $(\rho_{n}) $ and $Z(f_{n} + Ug_{n}+ Wh_{n})$   we  conclude that Img$(\rho_{n})= F_{n,\gamma}= Z(f_{n} + Ug_{n}+ Wh_{n})$. So, $ F_{n,\gamma}$ is a singular hypersurface of degree \lq$n$\rq \medspace  in $ \mathbb{P}^4$. 
\end{proof}
\begin{proposition}\label{proposition:proposition 4.5}
The Blow up of $F_{n, \gamma}$ along the line $\ell$  is isomorphic to $ \mathbb{P}(V_{n}) $.
\end{proposition}
\begin{proof} 
Similar to the proof of Theorem \ref{foo}.
\end{proof}
\begin{remark}
From Theorem \ref{theorem:theorem 4.4}  hypersurface $F_{n,\gamma} $ is singular along a line $ X_{0}=0, X_{1}=0, X_{2}=0$.  It is the same line $\ell$ which is the image of the  $(\tilde {t_{n}})_{0}$ under the map $ \rho_{n}$. 
 
\end{remark}

\begin{example}
In above Theorem \ref{theorem:theorem 4.4} if  we choose   $f_{n}=X_{0}^{n}$ and   $g_{n}=X_{1}^{n-1}$, $ h_{n}= X_{2}^{n-1}$ then image of $ \rho_{n}$ is $X_{0}^{n}+ UX_{1}^{n-1}+ WX_{2}^{n-1}$ which is not smooth along  the line $ X_{0}=0, X_{1}=0, X_{2}=0$ in $\mathbb{P}^4$.
\end{example}

\begin{proof}[\textbf{Proof of the Theorem \ref{raz}} ]
 The existence of the open set $U_{n}$ is similar to that of $ U$ in Theorem \ref{baz}.
 For each $ s\in U_{n}$  we see that \lq $s$\rq \medspace is a nowhere vanishing section of $ \mathcal{O}_{\mathbb{P}^2}(n)\Oplus \mathcal{O}_{\mathbb{P}^2}^{\bigoplus 2}(n-1)$, then cokernel of the following map $$ \mathcal{O}_{\mathbb{P}^2} \longrightarrow \mathcal{O}_{\mathbb{P}^2}(n)\Oplus \mathcal{O}_{\mathbb{P}^2}^{\bigoplus 2}(n-1)$$  \hspace{150pt} $1\longmapsto s $\\
 \\
is a vector bundle denoted by $V^{'}_{n,s}$ over $ \mathbb{P}^2$. Let $V_{n,s}=V^{'}_{n,s}(1-n)$  then second statement of the theorem is followed by  Remark \ref{remark:remark 4.3} and Theorem \ref{theorem:theorem 4.4}.
\end{proof}
\begin{remark}
We observe that we get a family of not-ample, nef, and big vector bundles over $ \mathbb{P}^2$ parameterized by an open set of  $ \mathbb{P}(H^{0}(\mathcal{O}_{\mathbb{P}^2}(n)\Oplus \mathcal{O}_{\mathbb{P}^2}^{\bigoplus 2}(n-1)))$ whose projectivization corresponds to a family of rational hypersurfaces of degree \lq $n$\rq \medspace containing a fixed line in $ \mathbb{P}^4$.
   
\end{remark}
\section{Acknowledgement}
I express my gratitude to my advisor, Professor D.S. Nagaraj, for suggesting this problem and for invaluable guidance and constructive suggestions throughout the course of this project. I also thank Dr. Souradeep Majumder, Supravat Sarkar, and Ashima Bansal for their helpful suggestions and insightful discussions.

\section{References}

\begin{enumerate}[label={[\arabic*]}]
    \item \label{1} E. Arbarello, M. Cornalba, P. A. Griffiths, J. Harris, {\it  Geometry of algebraic curves. Vol. I}, Grundlehren der 
mathematischen Wissenschaften 267 [Fundamental Principles of Mathematical Sciences]
Springer-Verlag, New York, 1985. 
    \item \label{2} W. Barth, {\it Some Properties of Stable rank-2 Vector Bundles on $\mathbb{P}^n$}, Math.
Ann. 226, 125–150 (1977).
\item \label{3} A.Bansal, S. Sarkar, S. Vats, { \it Isomorphism of Multiprojective Bundles and Projective Towers}, https://doi.org/10.48550/arXiv.2311.00999.

\item \label{4} K. Dan, D. S. Nagaraj, {\it Null correlation bundle on $\mathbb{P}^3$}, J. Ramanujan Math. Soc. 28A (2013), 75–80.

\item \label{5} D. Eisenbud, J. Harris, {\it 3264 and All That—A Second Course in Algebraic Geometry (2016)}
(Cambridge: Cambridge University Press).

  \item \label{6} A. El Mazouni,   D. S. Nagaraj, {\it Hyperplane sections of the projective bundle associated to the tangent bundle of $\mathbb{P}^2$} Geom. Dedicata 217 (2023). 
\item \label{7} W. Fulton, { \it (1984) Intersection Theory}, Springer-Verlag.
\item \label{8}  S. Galkin, D.S. Nagaraj {\it Projective bundles and blow-ups of projective spaces}, Ann. Mat. Pura Appl. (4)201(2022), no.6, 2707–2713.

\item\label{9} R. Hartshorne, {\it Algebraic Geometry (1977)} (NewYork: Springer-Verlag).
 
\item  \label{10} S.L. Kleiman, {\it Ample vector bundles on algebraic surfaces}
Proc. Amer. Math. Soc. 21 (1969), 673–676

\item \label{11} C.Okonek, M.Schneider, H.Spindler,\textit {Vector bundles on complex projective spaces}, Progress in Mathematics, 3. Birkhäuser, Boston, MA, 1980.

\item \label{12} N. Ray, {\it Examples of blown up varieties having projective bundle structures}, Proc. Indian Acad. Sci. Math. Sci. 130 (2020).

  \end{enumerate}

\vspace{3pt}
\begin{flushleft}
{\scshape Indian Institute of Science Education and Research Tirupati, Rami Reddy Nagar, Karakambadi Road, Mangalam (P.O.), Tirupati, Andhra Pradesh, India – 517507.}

{\fontfamily{cmtt}\selectfont
\textit{Email address: shivamvats@students.iisertirupati.ac.in} }
\end{flushleft}

\end{large}
\end{document}